\documentclass{amsart}
\usepackage{hyperref}

\newcommand*{\mailto}[1]{\href{mailto:#1}{\nolinkurl{#1}}}

\newtheorem{theorem}{Theorem}

\newtheorem{proposition}[theorem]{Proposition}
\newtheorem{corollary}[theorem]{Corollary}

\newcommand{\R}{{\mathbb R}}

\newcommand{\be}{\begin{equation}}
\newcommand{\ee}{\end{equation}}

\newcommand{\id}{{\rm 1\hspace{-0.6ex}l}}

\begin{document}

\title{On the Substitution Rule for Lebesgue--Stieltjes Integrals}

\author[N.\ Falkner]{Neil Falkner}
\address{The Ohio State University, Department of Mathematics\\
231 West 18th Avenue\\ Columbus, Ohio 43210\\ USA}
\email{\mailto{falkner@math.ohio-state.edu}}

\author[G.\ Teschl]{Gerald Teschl}
\address{Faculty of Mathematics\\ University of Vienna\\
Nordbergstrasse 15\\ 1090 Wien\\ Austria\\ and International
Erwin Schr\"odinger
Institute for Mathematical Physics\\ Boltzmanngasse 9\\ 1090 Wien\\ Austria}
\email{\mailto{Gerald.Teschl@univie.ac.at}}
\urladdr{\url{http://www.mat.univie.ac.at/~gerald/}}

\thanks{Expo. Math. {\bf 30}, 412--418 (2012)}
\thanks{{\it Teschl's research supported by the Austrian Science Fund (FWF)
under Grant No.\ Y330}}


\begin{abstract}
We show how two change-of-variables formul\ae\ for Lebesgue--Stieltjes
integrals generalize when all continuity hypotheses on the integrators
are dropped. We find that a sort of ``mass splitting phenomenon'' arises.
\end{abstract}

\maketitle

Let $M\colon[a,b]\to\R$ be increasing.%
\footnote{By ``increasing,'' we mean ``non-decreasing.''
Of course, $a$ and $b$ are real numbers with $a<b$.}
Then the measure corresponding to $M$ may be defined to be the unique
Borel measure $\mu$ on $[a,b]$ such that for each continuous function
$f\colon[a,b]\to\R$, the integral $\int_{[a,b]}f\,d\mu$ is equal to the
usual Riemann-Stieltjes%
\footnote{For an excellent exposition of Riemann--Stieltjes integration,
see \cite{atm} and \cite{ru}.}
integral $\int_a^b f(x)\,dM(x)$.
Now let $f\colon[a,b]\to\R$ be a bounded%
\footnote{Here and elsewhere in this paper, we have chosen to focus on
bounded integrands but our statements may be extended in the usual way
to suitable unbounded integrands.}
Borel function.
Then by definition, the Lebesgue-Stieltjes integral
$\int_a^b f(x)\,dM(x)$ is equal to $\int_{[a,b]}f\,d\mu$.
If $a<c<b$, then of course the equation
\[
\int_a^b f(x)\,dM(x)=\int_a^c f(x)\,dM(x)+\int_c^b f(x)\,dM(x)
\]
holds but to understand this properly, one should realize that
the point $c$ contributes $f(c)\mu(\{c\})=f(c)\bigl(M(c+)-M(c-)\bigr)$
to $\int_a^b f(x)\,dM(x)$ and this contribution is split into a
contribution of $f(c)\bigl(M(c)-M(c-)\bigr)$ to $\int_a^c f(x)\,dM(x)$
and a contribution of $f(c)\bigl(M(c+)-M(c)\bigr)$ to $\int_c^b f(x)\,dM(x)$.
This simple kind of splitting was pointed out by Stieltjes himself
(\cite{js}, pp.\ J70--J71, item 38; see also \cite{gb}, pp.\ 27--28, item 38)
and is closely related to the ``mass splitting phenomenon'' in
change-of-variables formul\ae\ alluded to in our abstract.

Now let $N\colon[M(a),M(b)]\to\R$ be increasing and let $\nu$ be the
measure on $[M(a),M(b)]$ corresponding to $N$.
Let $\Lambda=N\circ M$.
Then $\Lambda\colon[a,b]\to\R$ is also increasing.
Let $\lambda$ be the measure on $[a,b]$ corresponding to $\Lambda$.
It is natural to ask what relations exist between the measures
$\lambda$, $\mu$, and $\nu$.

If $N$ is continuous and $W$ is any generalized inverse%
\footnote{To say that $W$ is a generalized inverse for the increasing
function $M$ means that $W$ is an increasing function from $[M(a),M(b)]$
to $[a,b]$ and for each $y$ in the range of $M$, $W(y)$ is in the closure
of the interval $M^{-1}[\{y\}]$.
This concept, with or without this name, is well-established in the
literature.
For further information, see \cite{eh}.}
for the increasing function $M$, then it is not hard to show that $\lambda$
is the image of $\nu$ under $W$ or equivalently,%
\footnote{This equivalence is a standard result about images of measures
under measurable mappings.
See for instance \cite{du}, Theorem 1.6.9.
It is stated there for probability measures but that restriction
is inessential.}
that for each bounded Borel function $f\colon[a,b]\to\R$, we have
\be\label{eq1}
\int_a^b f(x)\,dN(M(x))=\int_{M(a)}^{M(b)} f(W(y))\,dN(y),
\ee
where $\int_a^b f(x)\,dN(M(x))$ means $\int_a^b f(x)\,d\Lambda(x)$.
In the special case where $N(y)\equiv y$,
this goes back to Lebesgue \cite{lb}.

If instead $M$ is continuous, then it is not hard to show that $\nu$ is the
image of $\lambda$ under $M$ or equivalently, that for each bounded Borel
function $g\colon[M(a),M(b)]\to\R$, we have
\be\label{eq2}
\int_a^b g(M(x))\,dN(M(x))=\int_{M(a)}^{M(b)}g(y)\,dN(y).
\ee
This is standard.%
\footnote{See for example  \cite{ry}, Chapter 1, \S 4, Proposition (4.10).
Attention is restricted there to the case where $N$ is right-continuous
but this is not essential.
In fact, if $M$ and $g$ are continuous, then \eqref{eq2} is obvious
by considering Riemann-Stieltjes sums, for then each upper Riemann-Stieltjes
sum for $\int_{M(a)}^{M(b)}g(y)\,dN(y)$ is equal in value to one of the
upper Riemann-Stieltjes sums for $\int_a^b g(M(x))\,dN(M(x))$, and similarly
for lower Riemann-Stieltjes sums, so the upper and lower Riemann-Stieltjes
integrals corresponding to $\int_a^b g(M(x))\,dN(M(x))$ lie between those
corresponding to $\int_{M(a)}^{M(b)}g(y)\,dN(y)$, so the Riemann-Stieltjes
integrals $\int_a^b g(M(x))\,dN(M(x))$ and $\int_{M(a)}^{M(b)}g(y)\,dN(y)$
are equal.
It follows that if $M$ is continuous and $g$ is a bounded Borel function,
then the Lebesgue-Stieltjes integrals $\int_a^b g(M(x))\,dN(M(x))$ and
$\int_{M(a)}^{M(b)}g(y)\,dN(y)$ are equal.
\endgraf
We would like to mention that change-of-variables formul\ae\ for certain
other types of integrals are given in \cite{le} and \cite{ta}.}
In the special case where $N(y)\equiv y$, this is attributed
in \cite{bo} (Vol.~I, Example 3.6.2) to Kolmogorov.

Our aim in this paper is to explain how \eqref{eq1} and \eqref{eq2}
generalize when no continuity assumptions are imposed on $M$ and $N$.
As we shall see, a key role is played by the left and right jumps
of $N$ at the points of the set
\[
H=\{ y\in[M(a),M(b)]:M^{-1}[\{y\}]\hbox{ contains more than one point}\}.
\]
We have chosen the letter $H$ for this set because it is the set
of all levels at which the graph of $M$ has a {\it horizontal\/} portion.
Note that $(M^{-1}[\{y\}])_{y\in H}$ is a pairwise disjoint
family of non-degenerate intervals in $[a,b]$.
Hence $H$ is countable.

Let $X$ and $\Xi$ be the left-continuous and right-continuous
generalized inverses for $M$.
These are the functions from $[M(a),M(b)]$ to $[a,b]$ defined
respectively by
\[
X(y)=\inf\{x\in[a,b]:y\le M(x)\}
\quad\hbox{and}\quad
\Xi(y)=\sup\{x\in[a,b]:M(x)\le y\}
\]
for all $y$ in $[a,b]$.
On $[M(a),M(b)]\setminus H$, we have $X=\Xi$,
while for each $y$ in the range of $M$, $X(y)$ is the left endpoint of
the interval $M^{-1}[\{y\}]$ and $\Xi(y)$ is its right endpoint.
It is easy to check that a function $W\colon[M(a),M(b)]\to\R$ is a generalized
inverse for $M$ if and only if $X\le W\le \Xi$.
In particular, $X$ and $\Xi$ are indeed generalized inverses for $M$.

\begin{proposition}
Suppose $N$ is right-continuous%
\footnote{By convention, we consider $N$ to be right-continuous
at $M(b)$ and we consider $N(M(b)+)$ to be $N(M(b))$.}
at $y$ for each $y$ in $H$.
Then $\lambda$ is the image of $\nu$ under $X$ and for each bounded
Borel function $f\colon[a,b]\to\R$, we have
\be\label{eq3}
\int_a^b f(x)\,dN(M(x))=\int_{M(a)}^{M(b)}f(X(y))\,dN(y).
\ee
\end{proposition}

\begin{proof}
It is easy to check that for each $x$ in $[a,b)$
and each $y$ in $[M(a),M(b)]$,
we have $X(y)\le x$ if and only if $y\le M(x+)$.
Let $G$ be the set of all $x$ in $[a,b)$ such that $M$ and $\Lambda$
are both right-continuous at $x$.
Then $[a,b]\setminus G$ is countable.
Hence $G$ is dense in $[a,b]$.
Let $x$ be in $G$.
Then
$\nu\bigl(X^{-1}\bigl[[a,x]\bigr]\bigr)
=\nu([M(a),M(x+)])
=\nu([M(a),M(x)])
=N(M(x)+)-N(M(a))$.
Now either for each $x'$ in $(x,b]$, we have $M(x)<M(x')$,
or there exists $x'$ in $(x,b]$ such that $M(x)=M(x')$.
Consider the case where for each $x'$ in $(x,b]$, we have $M(x)<M(x')$.
Then since $x$ is in $G$, $M(x)<M(x')\to M(x)$ as $x'\downarrow x$,
so $N(M(x'))\to N(M(x)+)$ as $x'\downarrow x$.
But again, since $x\in G$, $N(M(x'))=\Lambda(x')\to\Lambda(x)=N(M(x))$
as $x'\downarrow x$.
Hence $N(M(x)+)=N(M(x))$.
Now consider the case where there exists $x'$ in $(x,b]$
such that $M(x)=M(x')$.
Then $M=M(x)$ on $[x,x']$, so $M(x)$ is in $H$, so $N(M(x)+)=N(M(x))$
by assumption.
Thus in any case, $N(M(x)+)=N(M(x))$.
Therefore
$\nu\bigl(X^{-1}\bigl[[a,x]\bigr]\bigr)
=N(M(x))-N(M(a))
=\Lambda(x)-\Lambda(a)$.
But since $x$ is in $G$, $\Lambda(x)-\Lambda(a)=\lambda([a,x])$.
Thus $\lambda([a,x])=\nu\bigl(X^{-1}\bigl[[a,x]\bigr]\bigr)$.
This holds for each $x$ in $G$.
Let $\mathcal P$ be the set of all intervals of the form $[a,x]$
with $x\in G$.
Then $\mathcal P$ is a $\pi$-system on $[a,b]$
and since $G$ is dense in $[a,b]$,
$\mathcal P$ generates the Borel $\sigma$-field on $[a,b]$.
As we've just seen,
$\mathcal P$ is contained in the set $\mathcal L$ of all Borel sets
$E\subseteq[a,b]$ such that $\lambda(E)=\nu(X^{-1}[E])$.
Note that $[a,b]\in\mathcal L$ because $\lambda([a,b])=\Lambda(b)-\Lambda(a)
=N(M(b))-N(M(a))=\nu([M(a),M(b)])=\nu(X^{-1}[[a,b]])$.
Hence $\mathcal L$ is a $\lambda$-system on $[a,b]$.
(The $\lambda$ in $\lambda$-system does not refer to our measure $\lambda$.)
It follows that for each Borel set $E\subseteq[a,b]$,
$\lambda(E)=\nu(X^{-1}[E])$, by the $\pi$-$\lambda$ theorem.
(See, for instance, \cite{du}, Theorem A.1.4.)
In other words, $\lambda$ is the image of $\nu$ under $X$,
as claimed.
Equation \eqref{eq3} follows from this.
\end{proof}

Similarly, we have:

\begin{proposition}
Suppose $N$ is left-continuous%
\footnote{By convention, we consider $N$ to be left-continuous
at $M(a)$ and we consider $N(M(a)-)$ to be $N(M(a))$.}
at $y$ for each $y$ in $H$.
Then $\lambda$ is the image of $\nu$ under $\Xi$ and
for each bounded Borel function $f\colon[a,b]\to\R$, we have
\be\label{eq4}
\int_a^b f(x)\,dN(M(x))=\int_{M(a)}^{M(b)}f(\Xi(y))\,dN(y).
\ee
\end{proposition}

When no continuity condition is imposed on $N$, then $\lambda$
need not be the image of $\nu$ under any point mapping.
Instead, for each $y$ in $H$, the mass that $\nu$ assigns
to $\{y\}$ is split in $\lambda$ between the singletons
$\{X(y)\}$ and $\{\Xi(y)\}$.  This was alluded to above in
our abstract and is explained in detail in our main result:

\begin{theorem}
Let $N_1$ be the increasing function that is obtained from $N$
by removing the jumps that $N$ has at the points of $H$.
For each $y$ in $H$, let 
\[
\Delta N(y,-)=N(y)-N(y-)
\quad\text{and}\quad
\Delta N(y,+)=N(y+)-N(y)
\]
be the left and right jumps of $N$ at $y$ respectively.
Then for each bounded Borel function $f\colon[a,b]\to\R$, we have
\be\label{eq5}
\aligned
\int_a^b f(x)\,dN(M(x))
={}&\int_{M(a)}^{M(b)}f(X(y))\,dN_1(y)\\
 &+\sum_{y\in H}f(X(y))\Delta N(y,-)\\
 &+\sum_{y\in H}f(\Xi(y))\Delta N(y,+).\\
\endaligned
\ee
Furthermore,
$X$ may be replaced by $\Xi$ in the first term on the right in \eqref{eq5}.
\end{theorem}

\begin{proof}
For each $y$ in $H$, observe that
$\Delta N(y,+)\ge0$ and $\Delta N(y,-)\ge0$,
let
\[
N^y_-=\Delta N(y,-)\id_{[y,M(b)]}
\qquad\hbox{and}\qquad
N^y_+=\Delta N(y,+) \id_{(y,M(b)]},
\]
and observe that
$N^y_-$ is right-continuous and $N^y_+$ is left-continuous.
Let $N_2=\sum_{y\in H}N^y_-$ and $N_3=\sum_{y\in H}N^y_+$.
Note that these series converge uniformly on $[M(a),M(b)]$,
because $\sum_{y\in H}[\Delta N(y,-)+\Delta N(y,+)]=\nu(H)<\infty$.
By definition,
\[
N_1=N-N_2-N_3,
\]
so $N=N_1+N_2+N_3$.
Now $N_1$, $N_2$, and $N_3$ are increasing on $[M(a),M(b)]$,
$N_2$ is right-continuous, $N_3$ is left-continuous,
and for each $y\in H$, $N_1$ is continuous at $y$.
Let $\nu_1$, $\nu_2$, and $\nu_3$ be the measures
corresponding to $N_1$, $N_2$, and $N_3$ respectively.
Let $H^c=[M(a),M(b)]\setminus H$.
Then $X=\Xi$ on $H^c$.
Also, for each Borel set $E\subseteq[M(a),M(b)]$, we have
$\nu(H^c\cap E)=\nu_1(E)$ and $\nu(H\cap E)=\nu_2(E)+\nu_3(E)$.
Let $f\colon[a,b]\to\R$ be a bounded Borel function.
By \eqref{eq3} and \eqref{eq4},
\[
\int_a^b f(x)\,dN_1(M(x))
=\int_{M(a)}^{M(b)}f(X(y))\,dN_1(y)
=\int_{M(a)}^{M(b)}f(\Xi(y))\,dN_1(y).
\]
By \eqref{eq3},
\[
\int_a^b f(x)\,dN_2(M(x))
=\int_{M(a)}^{M(b)}f(X(y))\,dN_2(y)
=\sum_{y\in H}f(X(y))\Delta N(y,-).
\]
By \eqref{eq4},
\[
\int_a^b f(x)\,dN_3(M(x))
=\int_{M(a)}^{M(b)}f(\Xi(y))\,dN_3(y)
=\sum_{y\in H}f(\Xi(y))\Delta N(y,+).
\]
The result follows by addition.
\end{proof}

\begin{corollary}
Equation \eqref{eq1} still holds if $N$ is just continuous at each point of $H$.
In particular, if $M$ is strictly increasing,
then \eqref{eq1} holds with no continuity assumption on $N$.
\end{corollary}

\begin{proof}
If $N$ is continuous at each point of $H$, then 
the two sums on the right in \eqref{eq5} vanish, $N_1=N$,
$\nu(H)=0$, and if $W$ is any generalized inverse for $M$, then $X\le W\le\Xi$,
with equality on $[M(a),M(b)]\setminus H$.
If $M$ is strictly increasing, then $H$ is empty, so it is vacuously true
that $N$ is continuous at each point of $H$.
\end{proof}

\begin{corollary}
For each bounded Borel function $g$ on the range of $M$, we have
\be\label{eq6}
\aligned
\int_a^b g(M(x))\,dN(M(x))
={}&\int_{M(a)}^{M(b)}g(M(X(y)))\,dN_1(y)\\
 &+\sum_{y\in H}g(M(X(y)))\Delta N(y,-)\\
 &+\sum_{y\in H}g(M(\Xi(y)))\Delta N(y,+),\\
\endaligned
\ee
where the notation is as in the theorem.
Furthermore, $X$ may be replaced by $\Xi$ in the first term on the
right in \eqref{eq6}.
\end{corollary}

\begin{proof}
Let $f=g\circ M$ in \eqref{eq5}.
\end{proof}

Note that \eqref{eq6} is a generalization of \eqref{eq2},
because in the special case where $M$ is continuous,
it is clear that $M(X(y))=y=M(\Xi(y))$ for each $y$ in $[M(a),M(b)]$.

Since equations \eqref{eq5} and \eqref{eq6} are a bit complicated,
it is worth noting that they yield some simpler-looking inequalities
when $f$ and $g$ are monotone.
For each increasing function $f\colon[a,b]\to\R$ and for each $y$ in $H$,
we have $f(X(y))\le f(\Xi(y))$, so by \eqref{eq5},
\be\label{eq7}
\int_{M(a)}^{M(b)}f(X(y))\,dN(y)
\le\int_a^b f(x)\,dN(M(x))
\le\int_{M(a)}^{M(b)}f(\Xi(y))\,dN(y).
\ee
Let $g\colon[M(a),M(b)]\to\R$ be increasing
and let $f$ be the increasing function $g\circ M$.
If $M$ is left-continuous, then for each $y$ in $[M(a),M(b)]$,
we have $M(\Xi(y))\le y$, so from the right-hand inequality in \eqref{eq7},
we get
\be\label{eq8}
\int_a^b g(M(x))\,dN(M(x))\le\int_{M(a)}^{M(b)}g(y)\,dN(y).
\ee
If instead $M$ is right-continuous, then for each $y$ in $[M(a),M(b)]$,
we have $y\le M(X(y))$, so from the left-hand inequality in \eqref{eq7},
we get
\be\label{eq9}
\int_{M(a)}^{M(b)}g(y)\,dN(y)\le\int_a^b g(M(x))\,dN(M(x)).
\ee
If $g$ is decreasing rather than increasing,
then the inequalities \eqref{eq8} and \eqref{eq9} must be reversed.
To see this, just replace $g$ by $-g$.

A related inequality, in the special case where $g(x)\equiv x^n$,
is established by a different method in \cite{be}, where
it is applied to prove a Gronwall lemma for Lebesgue--Stieltjes integrals.
An application of \eqref{eq6} can be found in \cite{kst}.

Our results can easily be extended, with appropriate modifications,
to the case where $[a,b]$ is replaced by any interval $I$ and $[M(a),M(b)]$
is replaced by the smallest interval $J$ containing the range of $M$.

\bigskip
\noindent
{\bf Acknowledgments.}
The authors thank Jonathan Eckhardt, Fritz Gesztesy, Aleksey Kostenko, Erik Talvila, and Harald Woracek  for helpful discussions and
hints with respect to the literature.


\begin{thebibliography}{XX}
\bibitem{atm}
Tom M. Apostol, {\em Mathematical Analysis: A Modern Approach to Advanced
Calculus}, Addison-Wesley, Reading, Massachusetts, USA, 1957.
\bibitem{be}
C. Bennewitz, {\em Spectral asymptotics for Sturm--Liouville equations},
Proc. London Math. Soc. (3) {\bf 59} (1989), 294--338.
\bibitem{gb}
Garrett Birkhoff, editor, assisted by Uta Merzbach, {\em A Source Book
in Classical Analysis}, Harvard University Press, Cambridge, Massachusetts,
1973.
\bibitem{bo} 
V. I. Bogachev, {\em Measure Theory}, Vols.~I and II,
Springer-Verlag, Berlin, 2007.
\bibitem{du}
R. Durrett, {\em Probability: Theory and Examples}, 4th ed., Cambridge
University Press, Cambridge, 2010.
\bibitem{eh}
P. Embrechts and M. Hofert, {\em A note on generalized inverses}, Math. Methods Oper. Res. (to appear).
\bibitem{kst}
A. Kostenko, A. Sakhnovich, and G. Teschl, {\em Weyl--Titchmarsh theory
for Schr\"odinger operators with strongly singular potentials},
Int. Math. Res. Not. {\bf 2012}, 1699--1747 (2012).
\bibitem{le}
S. Leader, {\em Change of variables in Kurzweil-Henstock Stieltjes
integrals}, Real Analysis Exchange {\bf 29(2)} (2003/2004), 905--920.
\bibitem{lb}
H. Lebesgue, {\em Sur l'integrale de Stieltjes et sur les oper\'erations
fonctionelles lin\'eaires}, C. R. Acad. Sci., Paris {\bf 150}, 86--88 (1910).
\bibitem{ta}
E. Talvila, {\em The regulated primitive integral},
Illinois J. Math. {\bf 53} (2009), no. 4, 1187--1219.
\bibitem{ry}
D. Revuz and M. Yor, {\em Continuous Martingales and Brownian Motion},
3rd ed., Springer-Verlag, Berlin, 1999.
\bibitem{ru}
Walter Rudin, {\em Principles of Mathematical Analysis}, 3rd ed.,
McGraw-Hill, New York, 1976.
\bibitem{js}
J. Stieltjes, {\em Recherches sur les fractions continues},
Ann. Fac. Sci. Toulouse, S\'erie 1, tome {\bf 8} (1894), no. 4, J1--J122;
Oeuvres, II, 402-566.
\end{thebibliography}
\end{document}